\documentclass[12pt,reqno]{amsart}

\usepackage[usenames]{color}
\usepackage{amsmath,amssymb,pdfsync,verbatim,graphicx,epstopdf,enumerate,fancybox}
\usepackage{mathrsfs}
\usepackage{hyperref}
\usepackage{xcolor}
\usepackage[colorinlistoftodos]{todonotes}
\hypersetup{colorlinks=true,linkcolor=blue}
\usepackage[notref,notcite]{}

\numberwithin{equation}{section}
\everymath{\displaystyle}

\newtheorem{theorem}{Theorem}[section]
\newtheorem{lemma}[theorem]{Lemma}

\newtheorem{corollary}[theorem]{Corollary}

\newtheorem{remark}[theorem]{Remark}

\newcommand*{\nn}{\nonumber}
\newcommand*{\R}{\mathbb{R}}
\newcommand*{\Rn}{\mathbb{R}^n}
\newcommand*{\Hn}{\mathbb{H}^n}

\newcommand*{\divg}{\operatorname{div}}

\newcommand*{\al}{\alpha_\lambda}

\title{A Symmetry problem for some quasi-linear equations in Euclidean space}

\author[R. Dutta]{Ramya Dutta}
\address[Ramya Dutta]{Institut Camille Jordan, Universit\'e Claude Bernard Lyon 1, B\^atiment Braconnier, 21 avenue Claude Bernard, 69622 Villeurbanne Cedex, France}
\email{dutta@math.univ-lyon1.fr}

\author[P.-D. Thizy]{Pierre-Damien Thizy}
\address[Pierre-Damien Thizy]{Institut Camille Jordan, Universit\'e Claude Bernard Lyon 1, B\^atiment Braconnier, 21 avenue Claude Bernard, 69622 Villeurbanne Cedex, France}
\email{pierre-damien.thizy@univ-lyon1.fr}

\begin{document}
	
	\begin{abstract} We prove sharp asymptotic estimates for the gradient of positive solutions to certain nonlinear $p$-Laplace equations in Euclidean space by showing symmetry and uniqueness of positive solutions to associated limiting problems. 
	\end{abstract}
	
	
	\maketitle
	
	\section{Introduction}
	In this article we are concerned with symmetry and uniqueness results for two quasi-linear elliptic equations in Euclidean space under point-wise bounds for the solutions. The first is the positive entire solutions of the eigenvalue problem for the $p$-Laplace operator in Euclidean space $\Rn$ where $n \ge 2$, $1 < p < n$,
	\begin{align}\label{1eqn1}
		-\Delta_p v = -\lambda v^{p-1} \text{ in } \Rn.
	\end{align}
	By standard elliptic regularity theory \cite{DiBe}, \cite{Tol84} any weak solution $v \in W_{\text{loc}}^{1,p}(\Rn)$ of \eqref{1eqn1} is in fact $C^{1,\alpha}_{\text{loc}}(\Rn)$. We remark that when $\lambda < 0$, by a standard application of Picone's identity (see for example \cite[Theorem-2.1]{AH98}) it follows that there are no positive $W^{1,p}_{\text{loc}}(\Rn)$ solutions to \eqref{1eqn1}. When $\lambda = 0$, the solutions $v$ are called the $p$-harmonic functions and by a well known Liouville theorem the only positive $p$-harmonic functions are constant. We refer to \cite[Corollary-6.11]{HKM} for the Liouville theorem and for more details on $p$-harmonic functions in general. Henceforth we will only consider the case $\lambda > 0$.
	\\\\ The second is a two weight quasi-linear elliptic equation in the punctured space $\Rn \setminus \{0\}$ arising from the Euler-Lagrange equation of the Hardy's inequality in $\Rn$ with two radial weights. Namely we consider the positive weak solutions $u \in W^{1,p}_{\text{loc}}(\Rn \setminus \{0\})$ of the equation 
	\begin{align}\label{1eqnH1}
		-\divg\left(|x|^{-ap}|\nabla u|^{p-2} \nabla u\right) = \frac{\mu}{|x|^{(a+1)p}}u^{p-1} \text{ in } \Rn \setminus \{0\}
	\end{align}
	where, $n \ge 2$, $1 < p < n$, $a \in \R$, $-\infty < \mu \le \overline{\mu}_a$ and $\overline{\mu}_a := \left\vert\frac{n - (a+1)p}{p}\right\vert^p$. The number $\overline{\mu}_a$ is the best constant in Hardy's inequality when $a \neq \frac{n-p}{p}$ (see \cite[Proposition-1.1]{HK12} and also \cite{ACP05}, \cite{CKN}). Again by standard elliptic theory all positive weak solutions $u$ of \eqref{1eqnH1} are in $C^{1,\alpha}_{\text{loc}}(\Rn \setminus \{0\})$.
	\subsection{The eigenfunction equation in $\Rn$} \,
	\\ We are concerned with solutions to the problem \eqref{1eqn1} together with the point-wise bounds
	\begin{align}\label{1eqn2}
		\begin{split}
			& -\Delta_p v = -\lambda v^{p-1} \text{ in } \Rn \\ 
			& C_1 e^{\al \left<x,\xi\right>} \le v(x) \le C_2 e^{\al \left<x,\xi\right>}
		\end{split}
	\end{align}
	where, $C_1, C_2 > 0$ are positive constants, $\xi \in \mathbb{S}^{n-1}$ and $\al := \left(\frac{\lambda}{p-1}\right)^{1/p}$ when $\lambda > 0$. Note that $v(x) = Ce^{\al \left<x,\xi\right>}$ solves the equation \eqref{1eqn2}. The main result here is the $1$-dimensional symmetry and uniqueness (up to multiplication by a positive constant) of solutions to the problem \eqref{1eqn2}.
	\begin{theorem}\label{1thm1} Let, $n \ge 2$, $1 < p < n$ and $\lambda > 0$. Then any weak solution of the problem \eqref{1eqn2} has $1$-dimensional symmetry and is given by $v(x) = Ce^{\al \left<x,\xi\right>}$ for some $C > 0$.
	\end{theorem}
	\smallskip
	\noindent Problem of $1$-dimensional symmetry and monotonicity of positive solutions has been extensively studied in context of $p$-Laplacian. For symmetry results in half-space $\Rn_+$, we refer to \cite{Dancer13}, \cite{FS2}, \cite{ES20} and for results in this direction for full space $\Rn$, we refer to \cite{FS1}, \cite{FSV3} and references therein.
	\\\\ A similar problem was considered in \cite{RSanH} for the positive eigenfunctions in the real Hyperbolic space $\Hn$ of sectional curvature $-1$, where the moving plane technique was adapted appropriately to prove horospherical symmetry which lead to the uniqueness up to multiplication by a positive constant. A suitable adaptation of moving plane technique does not seem to be immediate for the Euclidean problem \eqref{1eqn2}. Instead we will use a variation of the strong maximum (or minimum) principle to show that the solutions of problem \eqref{1eqn2} have $1$-dimensional symmetry from which the uniqueness result follows. 
	\\\\ Problem \eqref{1eqn2} together with the point-wise bounds appear naturally as a limiting problem in various contexts. As a first application of Theorem-\ref{1thm1} we look at sharp asymptotic gradient estimates at infinity of positive solutions of $p$-Laplace equations with a non-linearity. In \cite{XH1} the authors consider the following problem
	\begin{align}\label{1eqn3}
		-\Delta_p u - \frac{\mu |u|^{p-2}u}{|x|^p} + \lambda |u|^{p-2}u = f(u) \text{ in } \Rn
	\end{align}
	for $u \in W^{1,p}(\Rn)$ where, $n \ge 2$, $1 < p < n$, $0 \le \mu < \overline{\mu} := \left(\frac{n-p}{p}\right)^p$, $\lambda > 0$ and $f \in C(\R)$ is a continuous function such that 
	\begin{align}\label{1eqn3a}
		\limsup_{t \to 0^+} \frac{|f(t)|}{|t|^{q-1}} \le A < +\infty, \quad \limsup_{t \to +\infty} \frac{|f(t)|}{|t|^{p^\ast-1}} \le A < +\infty
	\end{align} 
	with $1 < p < q < p^\ast = \frac{np}{n-p}$.
	\\ The sharp point-wise asymptotic estimates of the solutions and its gradient for the case $\lambda = 0$, $\mu = 0$ and $f(u) = |u|^{p^\ast - 2}u$, which corresponds to a quasi-linear critical Sobolev equation in $\Rn$, was first proved in \cite{Vetois}. In \cite[Theorem-1.3 and Theorem-1.6]{XH1} the following precise point-wise asymptotic estimate of positive solutions of \eqref{1eqn3} has been established using comparison theorems. Let $\gamma_1, \gamma_2$ be the two roots of the equation 
	\begin{align}\label{1eqn4}
		\gamma^{p-1}(n-p - (p-1)\gamma) = \mu
	\end{align} 
	such that $0 \le \gamma_1 < \frac{n-p}{p} < \gamma_2 \le \frac{n-p}{p-1}$ where, $\mu \in \left[0,\overline{\mu}\right)$, $\al := \left(\frac{\lambda}{p-1}\right)^{1/p}$ and $u \in W^{1,p}(\Rn)$ be any positive solution of \eqref{1eqn3} such that $f(u)$ is non-negative in $B_\rho(0)$ and $\Rn \setminus B_R(0)$ for some $\rho, R > 0$, then the following sharp point-wise asymptotic estimates hold
	\begin{align}\label{1eqn5}
		& u(x) \asymp |x|^{-\gamma_1}, \text{ for } x \in B_\rho(0) \setminus \{0\} \\
		\label{1eqn6}
		& u(x) \asymp |x|^{-\frac{n-1}{p(p-1)}}e^{-\al |x|}, \text{ for } x \in \Rn \setminus B_{R}(0)
	\end{align}
	where the constants of asymptotics depend implicitly on the solution $u$ as well as the parameters $n, p, \mu, q, A, f$ of the equation \eqref{1eqn3}. We use the notation $f \asymp g$ for positive functions $f,g$ to mean $C_1 g \le f \le C_2 g$ for some positive constants $C_1 , C_2 > 0$.
	\\\\ Sharp apriori point-wise asymptotic estimates on the general positive entire solutions and its gradient at infinity play a crucial role in problem of classification and proving radial symmetry of solutions to problems involving the $p$-Laplacian. The sharp asymptotic gradient estimate is especially relevant to the proof of radial symmetry of positive entire solutions using the moving plane technique. The upper bound on the gradient estimate is usually easier to prove and follows from the standard $C^{1,\alpha}$-estimates of \cite{DiBe}, \cite{Tol84}. However proving the gradient estimate from  below relies on an additional blow-up analysis and classification of solutions to the associated limiting problem. We refer to \cite{Sciunzi}, \cite{Vetois} for proof of radial symmetry of positive entire solutions in $\mathcal{D}^{1,p}(\Rn)$ to problem \eqref{1eqn3} with critical non-linearity $f(u) = u^{p^\ast - 1}$ when $\lambda = \mu = 0$ and to \cite{OSV} for $\lambda = 0$, $\mu \in (0,\overline{\mu})$. It is worth pointing out that when $\lambda = 0$ and $f(u) = u^{p^\ast - 1}$, the solutions have polynomial decay at infinity. In particular $u(x) \asymp |x|^{-\gamma}$ as $|x| \to \infty$ where $\gamma = \gamma_2$ is the larger root in \eqref{1eqn4}. In this case sharp behavior of the gradient at infinity $|\nabla u(x)| \asymp |x|^{-\gamma  - 1}$ is proved by classifying the blow-up limit of rescaling the solution $u_{R}(x) := R^{\gamma}u(Rx) \to u_\infty(x)$ in $\Rn \setminus \{0\}$ as $R \to \infty$ where the limit function $u_\infty$ develops an isolated singularity at origin. The limiting solution $u_\infty \in C^{1,\alpha}(\Rn \setminus \{0\})$ satisfies the equation 
	\begin{align}\label{1eqn8}
		\begin{cases} -\Delta_p u_\infty - \frac{\mu}{|x|^p}u_\infty^{p-1} = 0 \text{ in } \Rn \setminus \{0\} \\
			u_{\infty} \asymp |x|^{-\gamma}.
		\end{cases}
	\end{align}
	and the solutions are given by $u_\infty(x) = C|x|^{-\gamma}$ where, $\gamma$ is a root of \eqref{1eqn4}. More recently this technique of establishing sharp asymptotic gradient estimate by classification of isolated singularities has been extended to the anisotropic $p$-Laplacian in \cite{EMSV24}, where the classification result has been proved using a refined comparison principle. The classification of isolated singularity of $p$-harmonic functions ($1 < p \le n$ and $\mu = 0$) in punctured space $\Rn \setminus \{0\}$ follows from \cite{kv}. In case $\mu = \overline{\mu} = \left(\frac{n-p}{p}\right)^p$ the classification result was established in \cite{PS05}. We refer to \cite{CC22}, \cite{CFR} where the sharp asymptotic estimate of positive solutions and its gradient at infinity proved in \cite{Vetois2}, \cite{Vetois} was used to classify the solutions. We refer to \cite{PEsposito18}, \cite{PEsposito21} and also \cite{CXL24}, \cite{CXL22} in the anisotropic setting where the classification of isolated singularity of the $n$-Laplacian and respectively the anisotropic $n$-Laplacian in $\Rn \setminus \{0\}$ plays a vital role in the analysis and classification of solutions to the Liouville equation in $\Rn$.
	\\\\ Now coming back to the problem \eqref{1eqn3} when $\lambda > 0$, owing to the exponential decay of the positive solutions of equation at infinity, the nature of blow-up is different and the corresponding limiting problem is \eqref{1eqn2}. In \cite[Theorem-2.1]{XH1} for positive radial solutions $u(r)$ of \eqref{1eqn3} a precise asymptotic expansion of $\left(-\frac{u'(r)}{u(r)}\right)^{p-1}$ as $r \to +\infty$ has been established, the principal term of which is given by
	\begin{align}\label{1eqn7}
		\lim\limits_{r \to +\infty} \left(-\frac{u'(r)}{u(r)}\right)^{p-1} = \al^{p-1} = \frac{\lambda}{p-1}.
	\end{align} Applying Theorem-\ref{1thm1} we retrieve the asymptotic behavior of gradient at infinity as in \eqref{1eqn7} but for general positive solutions $u \in W^{1,p}(\Rn)$ of \eqref{1eqn3} (without the radiality assumption).
	\begin{theorem}\label{1thm2} 
		Let $n \ge 2$, $1 < p < n$, $0 \le \mu < \overline{\mu} := \left(\frac{n-p}{p}\right)^p$, $\lambda > 0$, $\gamma_1$ be the smaller root of \eqref{1eqn4} such that $\gamma_1 \in \left(0,\frac{n-p}{p}\right)$ and $u \in W^{1,p}(\Rn)$ be a positive solution of \eqref{1eqn3} with $f \in C(\R)$ satisfying \eqref{1eqn3a} such that $f(u)$ is non-negative in $B_{\rho_0}(0) \setminus \{0\}$ and $\Rn \setminus B_{R_0}(0)$ for some $R_0 > \rho_0 > 0$. Then 
		\begin{align}\label{1eqn9o}
			& \lim_{|x| \to 0^+} \frac{|x||\nabla u(x)|}{u(x)} = \gamma_1 \\
			\label{1eqn9}
			& \lim_{|x| \to +\infty} \frac{|\nabla u|}{u}(x) = \al
		\end{align}
		and this in particular implies $|\nabla u(x)| \asymp |x|^{-\gamma_1 - 1}$ in  $B_{\rho}(0) \setminus \{0\}$ and $|\nabla u(x)| \asymp |x|^{-\frac{n-1}{p(p-1)}}e^{-\al |x|}$ in $\Rn \setminus B_R(0)$ for some $R > \rho > 0$.	
	\end{theorem}
	\smallskip
	\noindent As a second application we look at a problem which is closely related to the Martin boundary problem coming from potential theory. When $p=2$ this corresponds to the problem of identifying the `ideal-boundary' of the operator $H_\lambda := (-\Delta + \lambda)$ in $\Rn$. We refer to \cite{Pin06} and references therein for a more detailed exposition of the problem from potential theoretic context. It is known that the ideal-boundary of operator $H_\lambda$ in $\Rn$ can be identified with the unit sphere $\mathbb{S}^{n-1}$. Positive solutions $v$ of $H_\lambda v = 0$ in $\Rn$ are given by the representation formula
	\begin{align}\label{1eqnrep}
		v(x) = \int_{\mathbb{S}^{n-1}} e^{\al\left<x,\xi\right>}\,d\mu(\xi)
	\end{align}
	where, $\mu$ is a positive measure on $\mathbb{S}^{n-1}$, in particular $v(x) = e^{\al\left<x,\xi\right>}$ corresponds to the dirac mass $\mu = \delta_{\xi}$ on $\mathbb{S}^{n-1}$. However such representation formulas are not available for $p$-Laplacian when $p \neq 2$. In case of the $p$-Laplacian we are able to get the following analogue of \cite[Example-2.14]{Pin06} as an immediate application of Theorem-\ref{1thm1}. 
	\begin{corollary}\label{corA}
		Let $n \ge 2$, $1 < p < n$, $\lambda > 0$ and $u$ be a positive solution to the problem
		\begin{align}\label{1eqnA}
			-\Delta_p u = -\lambda u^{p-1}
		\end{align}
		in an exterior domain $\Rn \setminus B_{r_0}(0)$ for some $r_0 > 0$ such that $u \in W^{1,p}(\Rn \setminus B_{r_0}(0))$. Then $u$ satisfies the point-wise asymptotic estimate $u(x) \asymp |x|^{-\frac{n-1}{p(p-1)}}e^{-\al |x|}$ in $\Rn \setminus B_{r_0}(0)$ and the limit 
		\begin{align}\label{1eqnB}
			P_\lambda(x,\xi) := \lim_{t \to +\infty} \frac{u(x - t\xi)}{u( - t\xi)}
		\end{align}
		exists and $P_\lambda(x,\xi) = e^{\al\left<x, \xi \right>}$ for $x \in \Rn$ and $\xi \in \mathbb{S}^{n-1}$.
	\end{corollary}
	
	\subsection{The two weight quasi-linear equation with Hardy potential in $\Rn \setminus \{0\}$}\,
	\\ We consider a class of quasi-linear elliptic Euler-Lagrange equations arising from the Hardy's inequality (see \cite[Proposition-1.1]{HK12} and also \cite{ACP05}, \cite{CKN}) with two radial weights. Let $n \ge 2$, $p \in (1,n)$, $a \in \R \setminus \left\{\frac{n-p}{p}\right\}$ and $\overline{\mu}_a := \left\vert\frac{n - (a+1)p}{p}\right\vert^p > 0$ denote the best constant in the Hardy's inequality
	\begin{align}\label{1eqnhardyineq}
		\overline{\mu}_a \int_{\Rn} \frac{|v|^{p}}{|x|^{(a+1)p}}\,dx \le \int_{\Rn} \frac{|\nabla v|^p}{|x|^{ap}}\,dx, \text{ for } v \in C_c^\infty\left(\Rn \setminus \{0\}\right).
	\end{align}
	We are considering the problem
	\begin{align}\label{peqn1}
		\begin{split}
			& -\Delta_{p,a} u - \frac{\mu}{|x|^{(a+1)p}}u^{p-1} = 0 \text{ in } \Rn \setminus \{0\} \\
			& C_1 |x|^{-\gamma} \le u(x) \le C_2|x|^{-\gamma} 
		\end{split}
	\end{align}
	where, $n \ge 2$, $1 < p < n$, $-\Delta_{p,a} u := -\divg\left(|x|^{-ap}|\nabla u|^{p-2}\nabla u \right)$, $a \in \R$, $-\infty < \mu  \le \overline{\mu}_a$, $C_1, C_2 > 0$ and $\gamma$ is a real root of the equation
	\begin{align}\label{peqnroot}
		|\gamma|^{p-2} \gamma (n - (a+1)p - (p-1)\gamma) = \mu.
	\end{align} 
	Note that we are also considering the case $a = \frac{n-p}{p}$. The radial function $u(x) = |x|^{-\gamma}$ solves \eqref{peqn1} if and only if $\gamma$ satisfies \eqref{peqnroot}. Now consider the auxiliary function $$f(\gamma) := |\gamma|^{p-2} \gamma (n - (a+1)p - (p-1)\gamma).$$ It attains its maximum at the unique point $\gamma_a^\ast := \frac{n - (a+1)p}{p}$, with maximum value $f(\gamma_a^\ast) = \left\vert\frac{n - (a+1)p}{p}\right\vert^p = \overline{\mu}_a$. So $f(\gamma) = \overline{\mu}_a$ has the unique real root $\gamma = \gamma_a^\ast$. Otherwise if $\mu < \overline{\mu}_a$, then $f(\gamma) = \mu$ has two distinct real roots $\gamma_{1,a}, \gamma_{2,a}$ such that $\gamma_{1,a} < \gamma_a^\ast < \gamma_{2,a}$. We summarize the placement of the roots $\gamma$ according to $\mu$ and $a$,
	\begin{align}\label{peqngamma}
		& 0 \le \gamma_{1,a} \le \gamma_a^\ast \le \gamma_{2,a} \le \frac{n-(a+1)p}{p-1}, \text{ when } \mu \in [0,\overline{\mu}_a] \text{ and } a < \frac{n-p}{p} \\ 
		& \gamma_{1,a} < 0 < \frac{n-(a+1)p}{p-1} < \gamma_{2,a}, \text{ when } \mu < 0 \text{ and } a < \frac{n-p}{p} \\
		& \gamma_{1,a} \le 0 \le \gamma_{2,a}, \text{ when } \mu \le 0 \text{ and } a = \frac{n-p}{p} \\
		& \frac{n-(a+1)p}{p-1} \le \gamma_{1,a} \le \gamma_a^\ast \le \gamma_{2,a} \le 0, \text{ when } \mu \in [0,\overline{\mu}_a] \text{ and } a > \frac{n-p}{p} \\
		& \gamma_{1,a} < \frac{n-(a+1)p}{p-1} < 0 < \gamma_{2,a}, \text{ when } \mu < 0 \text{ and } a > \frac{n-p}{p}.
	\end{align}
	The case $a = 0$ and $\mu \in [0,\overline{\mu}_a)$ corresponds to \eqref{1eqn8}. A complete classification and asymptotic behavior of the radial solutions to \eqref{peqn1} was given in \cite{IT21}.
	\begin{theorem}\label{1thmhardy} Let $n \ge 2$, $1 < p < n$, $a \in \R$, $-\infty < \mu  \le \overline{\mu}_a$, $\gamma$ is a real root of \eqref{peqnroot} and $u \in C^{1,\alpha}_{\text{loc}}(\Rn \setminus \{0\})$ be a solution to the problem \eqref{peqn1}. Then $u(x) = C|x|^{-\gamma}$ in $\Rn \setminus \{0\}$ for some $C > 0$.
	\end{theorem}
	\begin{remark}
		When $a \le \frac{n-p}{p}$ and $\gamma = 0$ ($\mu = 0$) then \eqref{peqn1} corresponds to bounded $\mathcal{A}$-harmonic functions $-\divg \mathcal{A}(x,\nabla u) := -\divg \left(|x|^{-ap}|\nabla u|^{p-2}\nabla u\right) = 0$ with a removable singularity at origin. We refer to \cite[Example-2.22]{HKM} the weight $\omega(x) := |x|^{-ap}$ in the stated range of $a \le \frac{n - p}{p}$ corresponds to the class of $p$-admissible weights and the $(p,\omega)$-capacity of the singleton set $\{0\}$ is zero, i.e., $\operatorname{Cap}_{p,\omega}(\{0\}) = 0$. Therefore by \cite[Theorem-7.36]{HKM} the solution $u$ has a removable singularity at origin and the result follows from the Liouville theorem \cite[Theorem-6.10]{HKM}. Theorem-\ref{1thmhardy} extends this result to the full range $a \in \R$. Note that when $a = \frac{n-p}{p}$, the solution $u(x) = -c\log |x|$ is a non-removable singularity of $-\divg\left(|x|^{p-n}|\nabla u|^{p-2}\nabla u\right) = 0$ in $\Rn \setminus \{0\}$, however, Theorem-\ref{1thmhardy} does not capture this.
	\end{remark}
	\smallskip
	\noindent As an immediate application of Theorem-\ref{1thmhardy} we have the sharp gradient estimates of solutions to the following quasi-linear elliptic Hardy-Sobolev-Maz'ya equation with critical exponent considered in \cite{PHT2023},
	\begin{align}\label{peqnHSM}
		\begin{split}
			& -\Delta_{p,a} u - \frac{\mu}{|x|^{(a+1)p}}u^{p-1} = \frac{u^{p_{a,b}^\ast - 1}}{|x|^{bp_{a,b}^\ast}} \text{ in } \Rn \\ 
			& \quad u > 0 \text{ and } u \in \mathcal{D}^{1,p}\left(\Rn, |x|^{-ap}\right)
		\end{split}
	\end{align}
	where, $0 \le a < \frac{n-p}{p}$, $a \le b < a+1$, the critical exponent $p_{a,b}^\ast := \frac{np}{n - (a + 1 - b)p} \in (p,p^\ast]$ and $0 < \mu < \overline{\mu}_a$. The space $\mathcal{D}^{1,p}\left(\Rn, |x|^{-ap}\right)$ denotes the completion of $C_c^1(\Rn)$ with respect to the semi-norm $$\lVert v \rVert_{\mathcal{D}^{1,p}\left(\Rn, |x|^{-ap}\right)} := \left(\int_{\Rn} |\nabla v|^p|x|^{-ap}\,dx\right)^{1/p}.$$
	\\ The sharp point-wise asymptotic estimates of the solutions and its gradient for the case $\mu = 0$, $-\infty < a < \frac{n-p}{p}$, which corresponds to a quasi-linear elliptic Caffarelli-Kohn-Nirenberg equation in $\Rn$, was first proved in \cite{Vetois2}. Let $u$ be a solution of \eqref{peqnHSM} then the following sharp point-wise asymptotic estimate was established in \cite{PHT2023},
	\begin{align}\label{peqndecay}
		u(x) \asymp |x|^{-\gamma_{1,a}} \text{ in } B_\rho(0)\setminus \{0\} \\
		u(x) \asymp |x|^{-\gamma_{2,a}} \text{ in } \Rn \setminus B_R(0)
	\end{align}
	for some $0 < \rho < R$ where, $\gamma_{1,a}, \gamma_{2,a}$ are the two positive real roots of \eqref{peqnroot} such that $0 < \gamma_{1,a} < \frac{n - (a+1)p}{p} < \gamma_{2,a} < \frac{n - (a+1)p}{p-1}$ (note that $0 < \mu < \overline{\mu}_a$ here). The constants of asymptotics depend implicitly on the solution $u$ as well as the parameters of the equation $n, p, \mu, a , b$. 
	\\ Using Theorem-\ref{1thmhardy} one can then proceed exactly as in the proof of \eqref{1eqn9o} in Theorem-\ref{1thm2} (which closely follows the proofs of \cite[Theorem-3.3]{OSV} and \cite[Theorem-1.4]{EMSV24}) to get the following sharp gradient estimates of the solutions to \eqref{peqnHSM}. We omit the proof of this result.
	\begin{theorem}\label{1thmhardygrad}
		Let $n \ge 2$, $1 < p < n$, $0 \le a < \frac{n-p}{p}$, $a \le b < a+1$, the critical exponent $p_{a,b}^\ast := \frac{np}{n - (a + 1 - b)p} \in (p,p^\ast]$ and $0 < \mu < \overline{\mu}_a$. Let $u$ be a solution of \eqref{peqnHSM} then it satisfies 
		\begin{align}\label{peqngraddlimit}
			\lim_{|x| \to 0^+} \frac{|x||\nabla u(x)|}{u(x)} = \gamma_{1,a} \\
			\lim_{|x| \to +\infty} \frac{|x||\nabla u(x)|}{u(x)} = \gamma_{2,a}
		\end{align}
		and hence the sharp point-wise asymptotic gradient estimates
		\begin{align}\label{peqngraddecay}
			|\nabla u(x)| \asymp |x|^{-\gamma_{1,a} - 1} \text{ in } B_\rho(0)\setminus \{0\} \\
			|\nabla u(x)| \asymp |x|^{-\gamma_{2,a} - 1} \text{ in } \Rn \setminus B_R(0)
		\end{align}
		for some $0 < \rho < R$ and the constants of asymptotics depend on the parameters $n, p, \mu, a , b$ and the solution $u$, where $\gamma_{1,a}, \gamma_{2,a}$ are the two positive real roots of \eqref{peqnroot} such that $0 < \gamma_{1,a} < \frac{n - (a+1)p}{p} < \gamma_{2,a} < \frac{n - (a+1)p}{p-1}$.
	\end{theorem}

	\section{Main ideas and motivation for the proof of Theorem-\ref{1thm1}}
	Before going to the proof of Theorem-\ref{1thm1} we briefly motivate the strategy of proof. We remark that a similar problem of classification of eigenfunctions was also considered in context of the real Hyperbolic space of constant sectional curvature $-1$ in \cite[Theorem-3.8]{RSanH} where the result was proved using a suitable modification of the moving plane method. A similar adaptation of moving plane technique does not seem to be immediate for the Euclidean problem \eqref{1eqn2}. Instead we will use a variation of the strong maximum (or minimum) principle to show that the solutions of problem \eqref{1eqn2} have $1$-dimensional symmetry which in turn implies the uniqueness result of Theorem-\ref{1thm1}.
	\\\\ It is interesting to note that the solutions of problem \eqref{1eqn2} correspond to the equality case in a gradient estimate with the best constant for positive eigenfunctions in $\Rn$. A gradient estimate with the best constant for positive eigenfunctions of $p$-Laplacian on complete non-compact manifolds $(M,g)$ with negative lower bound on Ricci curvature $\operatorname{Ric} \ge - (n-1)g$ was proved in \cite{SW14}. A straightforward adaptation of \cite[Theorem-1.1]{SW14} gives the best constant for the gradient estimate of positive eigenfunctions in $\Rn$. We add a detailed proof of the lemma in Appendix for the sake of completeness.
	\begin{lemma}\label{lemmagrad}
		Let $n \ge 2$, $1 < p < n$ and $\lambda > 0$. Suppose $v \in C^{1,\alpha}_{\text{loc}}(\Rn)$ to be a positive eigenfunction of the $p$-Laplacian corresponding to the eigenvalue $-\lambda$ i.e., 
		\begin{align}\label{E1}
			-\Delta_p v = -\lambda v^{p-1} \text{ in } \Rn.
		\end{align}
		Then $v$ satisfies the sharp gradient estimate
		\begin{align}\label{E2}
			|\nabla \log v| \le \al
		\end{align}
		where, $\al = \left(\frac{\lambda}{p-1}\right)^{1/p}$.
	\end{lemma}
	\noindent We start by considering the simple situation where equality is attained in the gradient estimate \eqref{E2} by a positive solution of \eqref{E1} at an interior point $x_0 \in \Rn$. 
	\smallskip
	\\ We use the notation $\mathbf{u} \otimes \mathbf{v} := \mathbf{u} \mathbf{v}^T$ that is $(\mathbf{u} \otimes \mathbf{v})_{ij} = (\mathbf{u}_i\mathbf{v}_j)$ for $\mathbf{u}, \mathbf{v} \in \Rn$ and $$\left<(\mathbf{u} \otimes \mathbf{v}) \mathbf{x}, \mathbf{y} \right> = \left<\mathbf{v},\mathbf{x}\right>\left<\mathbf{u},\mathbf{y}\right>$$
	for $\mathbf{x}, \mathbf{y} \in \Rn$. We denote the formally linearized $p$-Laplace operator by
	\begin{align}\label{2eqn1}
		L_v(f) := \divg \left(|\nabla v|^{p-2}A(\nabla f)\right)
	\end{align}
	where, $$A(\nabla f) := \left[\text{Id} + (p-2)|\nabla v|^{-2} (\nabla v \otimes \nabla v)\right](\nabla f) = \nabla f + (p-2)|\nabla v|^{-2} \left<\nabla v, \nabla f\right> \nabla v.$$ Note that when $|\nabla v| > 0$ we have
	\begin{align}\label{2eqn1a}
		\left<A(\nabla f), \nabla f\right> = |\nabla f|^2 + (p-2)|\nabla v|^{-2} \left<\nabla v, \nabla f \right>^2 \ge \min\{1,p-1\} |\nabla f|^2
	\end{align}
	so that $L_v$ is a strictly positive definite elliptic operator whenever $|\nabla v| > 0$. Suppose $v$ to be a $C^{1,\alpha}$ positive solution of \eqref{E1}, by standard elliptic regularity theory $v$ is smooth away from the critical point set of $v$ i.e., $v \in C^\infty(\Rn \setminus \mathcal{C}_v)$ where the set $\mathcal{C}_v := \{\nabla v = 0\}$ of critical points of $v$ is closed. Therefore, $v_{x_j} = \frac{\partial v}{\partial x_j}$ (for $j = 1, \cdots, n$) satisfies the linearized equation corresponding to \eqref{E1} weakly,
	\begin{align}\label{2eqn2}
		-L_v(v_{x_j}) = -(p-1)\lambda v^{p-2} v_{x_j}
	\end{align}
	for each $j = 1, \cdots, n$ in $\Rn \setminus \mathcal{C}_v$. That is 
	\begin{align}\label{2eqn2a}
		\int_{\Rn} |\nabla v|^{p-2} \left<\nabla v_{x_j}, \nabla \varphi \right> + (p-2)|\nabla v|^{p-4} \left<\nabla v, \nabla v_{x_j}\right> \left<\nabla v, \nabla \varphi \right> \,dx = -(p-1)\lambda \int_{\Rn} v^{p-2} v_{x_j} \varphi \,dx
	\end{align}
	for all $\varphi \in W^{1,2}(\Rn \setminus \mathcal{C}_v)$ with compact support. We also make a note of the simple fact that 
	\begin{align}\label{2eqn3}
		-L_v(v) = -(p-1)\lambda v^{p-2} v.
	\end{align}
	Therefore taking linear combination of \eqref{2eqn2} and \eqref{2eqn3} we get
	\begin{align}\label{2eqn4}
		-L_v(\left<\nabla v, \xi\right> + cv) = -(p-1)\lambda v^{p-2} (\left<\nabla v, \xi\right> + cv)
	\end{align}
	is satisfied weakly in $\Rn \setminus \mathcal{C}_v$ for any $\xi \in \Rn$ and $c \in \R$ .
	\begin{lemma}\label{lemmaeq}
		Let $v$ be a positive $C^{1,\alpha}$ solution of \eqref{E1} in $\Rn$ such that $|\nabla \log v|(x_0) = \al$ for some $x_0 \in \Rn$. Then $v(x) = Ce^{\al \left<x,\xi\right>}$ for some $C > 0$ and $\xi \in \mathbb{S}^{n-1}$.
	\end{lemma}
	\begin{proof}
		Let us denote $\frac{\nabla \log v(x_0)}{|\nabla \log v(x_0)|} = \xi \in \mathbb{S}^{n-1}$. Then from the gradient estimate of Lemma-\ref{lemmagrad} and \eqref{2eqn4}, we know $(\al v - \left<\nabla v, \xi\right>)$ is a non-negative weak solution of the linearized equation
		\begin{align}\label{2eqn5}
			-L_v(\al v - \left<\nabla v, \xi\right>) = -(p-1)\lambda v^{p-2} (\al v - \left<\nabla v, \xi\right>)
		\end{align}
		in $\Rn \setminus \mathcal{C}_v$ where, $\mathcal{C}_v = \{\nabla v = 0\}$. Furthermore, $(\al v - \left<\nabla v, \xi\right>)(x_0) = 0$ implies $|\nabla v(x_0)| > 0$ i.e., $L_v$ is strictly positive definite and elliptic in a neighborhood $B_{\delta}(x_0)$ for some $\delta > 0$. Applying Harnack inequality (see \cite{Serrin70}, \cite{T}) for \eqref{2eqn5} in $B_\delta(x_0)$ we get $(\al v - \left<\nabla v, \xi\right>) \equiv 0$ in $B_\delta(x_0)$. This in particular also implies $|\nabla v| > 0$ in $\overline{B_{\delta}(x_0)}$. Let $\Omega$ denote the connected component of $\Rn \setminus \mathcal{C}_v$ such that $x_0 \in \Omega$. Again applying Harnack inequality we have
		\begin{align}\label{2eqn6}
			\al v - \left<\nabla v, \xi\right> \equiv 0 \text{ in } \Omega.
		\end{align} 
		Then $|\nabla v| > 0$ on $\partial \Omega$ which leads to a contradiction unless $\partial \Omega = \emptyset$ i.e., $\Omega = \Rn$. But from the gradient estimate $|\nabla v| \le \al v$ we get $\left<\nabla v, \xi'\right> \equiv 0$ for any $\xi' \in \xi^{\perp}$ i.e., $\left<\xi',\xi\right> = 0$. This implies $v$ is only a function of $\left<x,\xi\right>$. This combined with $\al v - \left<\nabla v, \xi\right> \equiv 0 \text{ in } \Rn$ from \eqref{2eqn6} gives the desired conclusion $v(x) = Ce^{\al \left<x,\xi\right>}$.
	\end{proof}
	\smallskip
	\noindent In light of the inequality \eqref{E2} and Lemma-\ref{lemmaeq} it is natural to expect rigidity for positive solutions $v$ of \eqref{E1} among the class of functions 
	\begin{align}\label{E3}
		\log v = \al E + b
	\end{align} where, $E$ satisfies $|\nabla E| \equiv 1$ in $\Rn$ and $b$ is a bounded function in $\Rn$. The only entire solutions of $|\nabla E| \equiv 1$ in $\Rn$ are given up to additive constants by the affine functions $E(x) = E(x,\xi) + C$ where,
	\begin{align}\label{E4}
		E(x,\xi) := \left<x,\xi\right>
	\end{align}
	(see for instance \cite[Lemma-7.1]{Crandall2008} and \cite{CC10}). Note that is the Busemann function in $\Rn$ corresponding to $\xi \in \mathbb{S}^{n-1}$ is precisely the affine function 
	\begin{align}\label{E4a}
		\lim_{t \to +\infty} (t - |x - t\xi|) = \left<x,\xi\right>.
	\end{align} 
	Positive solutions $v$ of \eqref{E1} in class \eqref{E3} are necessarily extremals of the inequality \eqref{E2}, i.e., $\sup_{\Rn} |\nabla \log v| = \al$ but the boundedness assumption on $b$ or equivalently the point-wise bounds on $v$ assumed in \eqref{1eqn2} does not immediately imply that equality in \eqref{E2} happens at an interior point $x_0 \in \Rn$. We show that the boundedness assumption on $b$ is sufficient to establish $1$-dimensional symmetry of the solutions which in turn implies that equality must hold in \eqref{E2} at in interior point $x_0 \in \Rn$ i.e., we must have $|\nabla \log v| \equiv \al$ in $\Rn$ and $b$ must be a constant function. Although the best constant in the inequality \eqref{E2} of Lemma-\ref{lemmagrad} is not directly used in the proof of Theorem-\ref{1thm1}, it is implicitly being used through the point-wise bounds assumed in \eqref{1eqn2}. We will also see this in proof of Theorem-\ref{1thmhardy}, where the best constant of the relevant gradient estimate \eqref{peqng} is not known a priori, but the point-wise bounds on the solutions are sufficient for the purpose.
	\section{Proof of the main results}
	\noindent We now proceed with the proof of Theorem-\ref{1thm1}.
	\begin{proof}[Proof of Theorem-\ref{1thm1}]
		
		Note that by making an orthogonal change of coordinates, without loss of generality we may assume $\xi = e_n \in \mathbb{S}^{n-1}$. Therefore we have 
		\begin{align}\label{2eqn7}
			-\Delta_p v &= -\lambda v^{p-1} \text{ in } \Rn \\ 
			\label{2eqn7a}
			C_1 e^{\al x_n} &\le v(x) \le C_2 e^{\al x_n}.
		\end{align}
		We will use the point-wise bounds \eqref{2eqn7a} to show that $v$ has $1$-dimensional symmetry i.e., it is a function of the $x_n$ variable. The uniqueness (up to multiplication be a positive constant) will follow from this. By translation invariance of \eqref{2eqn7}, $C^{1,\alpha}$-estimate (\cite[Theorem-1]{Tol84}, \cite{DiBe}) and Harnack inequality \cite{T} we have
		$$|\nabla v|(x_0) \le \sup_{B_{1/2}(x_0)} |\nabla v| \le C\sup_{B_{1}(x_0)} v \le c(n,p,\lambda) v(x_0)$$ for all $x_0 \in \mathbb{R}^n$ for some universal constant $c(n,p,\lambda) > 0$ i.e., $|\nabla \log v| \le c(n,p,\lambda)$ in $\Rn$.
		\\ Let us denote $w := \log v$, then from \eqref{2eqn7} and \eqref{2eqn7a} we get $w$ satisfies
		\begin{align}\label{2eqn9}
			& -\Delta_p w = -\lambda + (p-1)|\nabla w|^p \text{ in } \Rn \\ 
			\label{2eqn9a}
			& \al x_n + c_1 \le w(x) \le \al x_n + c_2
		\end{align}
		for $c_1, c_2 \in \R$. For $\nu \in \mathbb{S}^{n-1}_+ := \mathbb{S}^{n-1} \cap \{x _n > 0\}$, we denote $w_\nu := \left<\nabla w, \nu\right>$ and since $|\nabla w| = |\nabla \log v| \le c(n,p,\lambda)$ we have
		\begin{align}\label{2eqns1}
			-\infty < \beta(\nu) := \inf_{\Rn} w_\nu \le \sup_{\Rn} w_\nu := \alpha(\nu) < +\infty.
		\end{align}
		We prove the result in two key steps.
		\\ \textbf{Step-I:} We claim that in \eqref{2eqns1} we have
		\begin{align}\label{2eqns2}
			\alpha(\nu), \beta(\nu) \in \{0,\alpha_\lambda\left<\nu,e_n\right>\}
		\end{align}
		for all $\nu\in \mathbb{S}^{n-1}_+$. We prove the claim for $\beta(\nu)$, the proof for $\alpha(\nu)$ follows similarly. To see this we fix $\nu$. If $\beta(\nu) = 0$ then we are through. Suppose now that $\beta(\nu) \neq 0$, then we show that $\beta(\nu) = \alpha_\lambda\left<\nu,e_n\right>$.
		\\ Let $\left\{y^{(k)}\right\}_{k \in \mathbb{N}}$ be a sequence of points in $\Rn$ such that $w_\nu(y^{(k)}) \to \beta(\nu)$ as $k \to \infty$. Define the function $w^{(k)}(x) := w(x + y^{(k)}) - \al y^{(k)}_n$. As before we note that $w^{(k)}$ satisfies the equation 
		\begin{align}\label{2eqn13}
			\begin{cases} 
				& -\Delta_p w^{(k)} = -\lambda + (p-1)|\nabla w^{(k)}|^p \\ 
				& \al x_n + c_1 \le w^{(k)}(x) \le \al x_n + c_2 \\
				& w^{(k)}_{\nu}(0) = w_{\nu}(y^{(k)}) \to \beta(\nu). 
			\end{cases}
		\end{align}
		Furthermore, from definition we have $\inf_{\Rn} w^{(k)}_{\nu} = \inf_{\Rn} w_\nu = \beta(\nu)$ for all $k \in \mathbb{N}$. Applying $C^{1,\alpha}$-estimate (\cite[Theorem-1]{Tol84}, \cite{DiBe}) we have $[w^{(k)}]_{C^{1,\alpha}(B_R)} \le C_R$ for some constant $C_R > 0$. Letting $R \to \infty$ and passing through a subsequence (which we continue to index with $k$) we have $w^{(k)} \to w^{(\infty)}$ in $C^1_{\text{loc}}(\Rn)$ such that $w^{(\infty)}$ also satisfies \eqref{2eqn9} and \eqref{2eqn9a}. Owing to $C^1_{\text{loc}}(\Rn)$ convergence we have $\inf_{\Rn} w^{(\infty)}_{\nu} \ge \beta(\nu)$ and $$\inf_{\Rn} w^{(\infty)}_{\nu} = w^{(\infty)}_{\nu}(0) = \beta(\nu)$$ i.e., the infimum is attained for $w^{(\infty)}_{\nu}$ in the interior of $\Rn$. 
		\\ Let us denote by $v^{(\infty)} = e^{w^{(\infty)}}$ which also satisfies \eqref{2eqn7} and \eqref{2eqn7a}. Denoting $v^{(\infty)}_{\nu} := \left<\nabla v^{(\infty)} , \nu\right>$ and using \eqref{2eqn4} we have $(v_\nu^{(\infty)} - \beta(\nu) v^{(\infty)})$ is a non-negative weak solution of the linearized equation 
		\begin{align}\label{2eqn14}
			-L_{v^{(\infty)}}(v^{(\infty)}_\nu - \beta(\nu) v^{(\infty)}) = -(p-1)\lambda \left(v^{(\infty)}\right)^{p-2}(v^{(\infty)}_\nu - \beta(\nu) v^{(\infty)})
		\end{align}
		in $\Rn \setminus \mathcal{C}_{v^{(\infty)}}$ where, $\mathcal{C}_{v^{(\infty)}} = \{\nabla v^{(\infty)} = 0\}$. Since, $\beta(\nu) \neq 0$ and $(v^{(\infty)}_\nu - \beta(\nu) v^{(\infty)})(0) = 0$, this implies there are no critical points of $v^{(\infty)}$ in an open neighborhood of the origin. Arguing similarly to the proof of Lemma-\ref{lemmaeq}, applying Harnack inequality (\cite{Serrin70}, \cite{T}) we get $(v^{(\infty)}_\nu - \beta(\nu) v^{(\infty)}) \equiv 0$ in this neighborhood and we get from a connectedness argument
		\begin{align}\label{2eqn14a}
			(v^{(\infty)}_\nu - \beta(\nu) v^{(\infty)}) \equiv 0 \text{ in } \Rn
		\end{align} 
		i.e., $w_\nu^{(\infty)} \equiv \beta(\nu)$ in $\Rn$. Since $\frac{d}{dt} w^{(\infty)}(\nu t) = w^{(\infty)}_\nu(\nu t) \equiv \beta(\nu)$, we have \begin{align}\label{2eqn14b}
			w^{(\infty)}(\nu t) = w^{(\infty)}(0) + \beta(\nu) t, \text{ for } t \in \R.
		\end{align}
		However, from the point-wise bounds \eqref{2eqn9a} of $w^{(\infty)}$ we also have 
		\begin{align}\label{2eqn14c}
			\al \left<\nu , e_n\right> t + c_1 \le w^{(\infty)}(\nu t) \le \al \left<\nu , e_n\right> t + c_2, \text{ for } t \in \R.
		\end{align}
		Comparing \eqref{2eqn14b} and \eqref{2eqn14c} we get $\beta(\nu) = \al \left<\nu , e_n\right>$, completing the proof of the claim.
		\\\\ \textbf{Step-II:} Finally with the previous claim we note that $w_\nu \ge 0$ for all $\nu \in \mathbb{S}^{n-1}_+$ as $\alpha_\lambda\left<\nu,e_n\right> > 0$. By continuity we must also have $w_\nu \ge 0$ for all $\nu \in \partial \mathbb{S}^{n-1}_+ = \mathbb{S}^{n-1} \cap \{\left<\nu, e_n\right> = 0 \}$ as well. But this implies $w_{x_j} \equiv 0$ for all $j = 1, \cdots, n-1$ and $w_{x_n} \ge 0$ i.e., $w$ is a non-decreasing function of the single variable $x_n$ and consequently so is $v$. By abuse of notation, let us denote $v(x) = v(x',t) = v(t)$ where $x = (x',t) \in \R^{n-1} \times \R$. 
		
		\smallskip
		
		\noindent We claim that $v(t) = Ce^{\al t}$ for some $C > 0$. Note that $v$ satisfies the ordinary differential equation
		\begin{align}\label{2eqn15} 
			-(|v'|^{p-2}v')' = -\lambda v^{p-1}.
		\end{align}
		Since, $v' \ge 0$ we can write \eqref{2eqn15} as 
		\begin{align}\label{2eqn16}
			((v')^{p-1})' = \lambda v^{p-1}
		\end{align}
		and using the asymptotic behavior $C_1 e^{\al t} \le v(t) \le C_2 e^{\al t}$ and its derivative $0 \le v'(t) \le Ce^{\al t}$, we integrate \eqref{2eqn16} from $-\infty$ to $t$ to get
		\begin{align}\label{2eqn17}
			(v'(t))^{p-1} = \lambda \int_{-\infty}^{t} v^{p-1}(s)\,ds \ge C_1^{p-1}\lambda\int_{-\infty}^{t} e^{(p-1)\al s}\,ds = \frac{C_1^{p-1}\lambda}{(p-1)\al} e^{(p-1)\al t} \ge C' v(t)^{p-1}
		\end{align} for some $C' > 0$. Thus in particular with $\nu = e_n$ in \eqref{2eqns1} we have $$0 < \beta(e_n) = \inf_{t \in \R} \frac{v'}{v} \le \sup_{t \in \R} \frac{v'}{v} = \alpha(e_n)$$
		i.e., $\alpha(e_n)$ and $\beta(e_n)$ are always non-zero. Therefore from Step-I, \eqref{2eqns2} we get $\beta(e_n) = \alpha(e_n) = \al$. Consequently we have $\frac{v'(t)}{v(t)} \equiv \al$, i.e., $v(t) = Ce^{\al t}$ for some $C > 0$, completing proof of the theorem.
	\end{proof}
	\smallskip
	\noindent We now present the proofs of Theorem-\ref{1thm2} and Corollary-\ref{corA} as applications of Theorem-\ref{1thm1} and Theorem-\ref{1thmhardy}.
	
	\begin{proof}[Proof of Theorem-\ref{1thm2}] We start by proving \eqref{1eqn9}. Using the point-wise decay estimate \eqref{1eqn6} of $u$ and the fact that $\frac{|f(u)|}{u^{p-1}} \le C(u^{q-p} + u^{p^\ast - p})$, we may rewrite equation \eqref{1eqn3} as
		\begin{align}\label{deqn1}
			-\Delta_p u = \left(- \lambda + \frac{\mu}{|x|^p} + \frac{f(u)}{u^{p-1}}\right) u^{p-1} = \left(- \lambda + \frac{\mu}{|x|^p} + O\left(|x|^{-\frac{n-1}{p(p-1)}}e^{-\al|x|}\right)^{q-p}\right) u^{p-1}
		\end{align}
		for $x \in \Rn \setminus B_{R_0}$ for some $R_0 > 0$ large.
		Let $\{x_k\}_{k \in \mathbb{N}}$ be a sequence of points in $\Rn \setminus B_{R_0}$ such that $|x_k| \to \infty$ as $k \to \infty$ and let us denote $\xi_k := \frac{x_k}{|x_k|} \in \mathbb{S}^{n-1}$. Define the function $u_k(x) := \frac{u(x - x_k)}{u(-x_k)}$ which satisfies
		\begin{align}\label{deqn2} 
			-\Delta_p u_k = \left(- \lambda + \frac{\mu}{|x - x_k|^p} + O\left(|x - x_k|^{-\frac{n-1}{p(p-1)}}e^{-\al|x - x_k|}\right)^{q-p}\right) u_k^{p-1} = (-\lambda + o(1))u_k^{p-1}
		\end{align}
		in $B_R$ as $k \to \infty$ for any $R > 0$ such that $B_R(x_k) \subset \Rn \setminus B_{R_0}$. 
		\\ Also by the point-wise decay estimate \eqref{1eqn6} of $u$ we have the bounds
		\begin{align}\label{deqn3}
			u_k(x) \asymp \left(\frac{|x - x_k|}{|x_k|}\right)^{-\frac{n-1}{p(p-1)}}e^{-\al (|x - x_k| - |x_k|)} \asymp e^{\al \left<x,\xi_k\right>}
		\end{align}
		in $B_R$ with uniform constants of asymptotics for all $k \ge K_R$ for some $K_R$ chosen large depending on $R$. Here we used the fact that $$|x_k| - |x - x_k| = \frac{|x_k|^2 - |x - x_k|^2}{|x_k| + |x - x_k|} = \frac{2}{1 + \frac{|x - x_k|}{|x_k|}}\left(\left<x,\xi_k\right> - \frac{|x|^2}{2|x_k|}\right) = \left<x,\xi_k\right> + o(1)$$ as $k \to \infty$.
		\\ Using $C^{1,\alpha}$-estimates \cite{DiBe}, \cite{Tol84} to \eqref{deqn2} and noting the locally uniform bounds in \eqref{deqn3} we have $[u_k]_{C^{1,\alpha}(B_R)} \le C_R$ for some $C_R > 0$ depending only on $R > 0$. Then letting $R \to \infty$ and passing through a subsequence (which we continue to index with $k$) we get $\xi_k = \frac{x_k}{|x_k|} \to \xi \in \mathbb{S}^{n-1}$ and $u_k \to u_\infty$ in $C^1_{\text{loc}}(\Rn)$ as $k \to \infty$ such that $u_\infty$ satisfies 
		\begin{align}\label{deqn4}
			-\Delta_p u_\infty = -\lambda u_\infty^{p-1} \text{ in } \Rn.
		\end{align} 
		Furthermore from \eqref{deqn3} we have the point-wise bounds 
		\begin{align}\label{deqn5}
			c_1 e^{\al \left<x,\xi\right>} \le u_{\infty}(x) \le c_2 e^{\al \left<x,\xi\right>}.
		\end{align}
		Then by Theorem-\ref{1thm1} we have $u_\infty(x) = e^{\al \left<x,\xi\right>}$ as $u_\infty(0) = u_k(0) = 1$ for all $k \in \mathbb{N}$. By $C^1_{\text{loc}}(\Rn)$ convergence of $u_k \to u_\infty$ we get $|\nabla u_k|(0) \to |\nabla u_\infty|(0) = \al$. Therefore for every sequence $x_k \to \infty$ we have a subsequence such that $\frac{|\nabla u(-x_k)|}{u(-x_k)} \to \al$ as $k \to \infty$. Hence the limit \eqref{1eqn9} exists and equals $\al$. The sharp asymptotic gradient estimate of $u$ at infinity now follows from this limit and the point-wise bounds \eqref{1eqn6}. 
		
		\bigskip
		
		\noindent Now to prove \eqref{1eqn9o} one can proceed exactly as in the proofs of \cite[Theorem-3.3]{OSV} and \cite[Theorem-1.4]{EMSV24}. Let $\{x_k\}_{k \in \mathbb{N}}$ be a sequence of points in $\Rn \setminus \{0\}$ such that $R_k := |x_k| \to 0^+$. We consider the rescaling of the solution $u_k(x) := \frac{u(R_k x)}{u(x_k)}$ for $k \ge 1$. 
		\\ Using the sharp point-wise asymptotic estimates \eqref{1eqn5}, for any $R > 1$ we have 
		\begin{align}\label{reqn1}
			C_1|x|^{-\gamma_1} \le u_k(x) \le C_2|x|^{-\gamma_1} \text{ in } A(R) := B_{R}(0) \setminus B_{1/R}(0) 
		\end{align} 
		uniformly for all $k \ge K$ for some $K \in \mathbb{N}$ chosen large enough depending on $R$. 
		\\ Also from \eqref{1eqn3a} we have $\frac{|f(u)|}{u^{p-1}} \le C(u^{q-p} + u^{p^\ast - p})$. Since, $0 < \gamma_1 < \frac{n-p}{p}$ we have $p - (q - p)\gamma_1 \ge p - (p^\ast - p)\gamma_1 > 0$. From the point-wise behavior of $u$, we get $u_k$ satisfies the equation
		\begin{align}\label{reqn2}
			-\Delta_p u_k - \frac{\mu}{|x|^p}u_k^{p-1} &= -\lambda R_k^p u_k^{p-1} + R_k^pO\left(u^{q-p}(R_kx) + u^{p^\ast - p}(R_kx)\right)u_k^{p-1} \nn \\ 
			&= -\lambda R_k^p u_k^{p-1} + R_k^{p - (p^\ast-p)\gamma_1}O\left(|x|^{-(q-p)\gamma_1} + |x|^{-(p^\ast-p)\gamma_1}\right)u_k^{p-1} = o(1)
		\end{align}
		in $A(R) = B_{R}(0) \setminus B_{1/R}(0) $ as $k \to +\infty$.
		\\ Therefore using the $C^{1,\alpha}$-estimates \cite{DiBe}, \cite{Tol84} to \eqref{reqn2} together with the point-wise bounds \eqref{reqn1} we have $[u_k]_{C^{1,\alpha}(A(R))} \le C_R$ for some $C_R > 0$. Letting $R \to \infty$ we may extract a subsequence (which we continue to index with $k$) such that $\xi_k := \frac{x_k}{R_k} \to \xi \in \mathbb{S}^{n-1}$ and $u_{k} \to u_\infty$ in $C^1_{\text{loc}}(\Rn \setminus \{0\})$, where $u_{\infty}$ satisfies the limiting problem
		\begin{align}\label{reqn3}
			\begin{split}
				& -\Delta_p u_{\infty} - \frac{\mu u_{\infty}^{p-1}}{|x|^p} = 0 \text{ in } \Rn \setminus \{0\} \\
				& u_\infty(x) \asymp |x|^{-\gamma_1}.
			\end{split}
		\end{align} Using Theorem-\ref{1thmhardy} we have $u_\infty(x) = |x|^{-\gamma_1}$ as $u_\infty(\xi) = \lim_{k \to \infty} u_k(\xi_k) = 1$. In particular for every sequence $x_k \to 0$ we may extract a subsequence such that
		\begin{align}
			\frac{R_k|\nabla u(x_k)|}{u(x_k)} = |\nabla u_k(\xi_k)| \to |\nabla u_\infty(\xi)| = \gamma_1.
		\end{align} Hence the limit \eqref{1eqn9o} exists and equals $\gamma_1$. The sharp asymptotic gradient estimate of $u$ near origin now follows from the point-wise bounds \eqref{1eqn5}. 
	\end{proof}
	\bigskip
	\begin{proof}[Proof of Corollary-\ref{corA}]
		Since, $u \in W^{1,p}(\Rn \setminus B_{r_0})$, we note that from the explicit sub and super-solution to the problem \eqref{1eqnA} constructed in proof of \cite[Theorem-1.6]{XH1} and the weak comparison theorems we have $$u(x) \asymp |x|^{-\frac{n-1}{p(p-1)}}e^{-\al |x|}$$ in $\{|x| > r_0\}$. Define the function $v_k(x) := \frac{u(x - t_k\xi)}{u(- t_k\xi)}$, where $t_k \to \infty$ as $k \to \infty$. It satisfies the equation 
		\begin{align}\label{ceqn1}
			-\Delta_p v_k = -\lambda v_k^{p-1} \text{ in } \Rn \setminus B_{r_0}(t_k\xi).
		\end{align}
		Fix $R > 0$, from the explicit point-wise bounds on $u$ we have 
		\begin{align}\label{ceqn2}
			C_1\left(\frac{|x - t_k\xi|}{|t_k\xi|}\right)^{-\frac{n-1}{p(p-1)}}e^{-\al (|x - t_k\xi| - |t_k\xi|)} \le v_k(x) \le C_2\left(\frac{|x - t_k\xi|}{|t_k\xi|}\right)^{-\frac{n-1}{p(p-1)}}e^{-\al (|x - t_k\xi| - |t_k\xi|)} 
		\end{align} 
		uniformly in $x \in B_R$ and $k \ge K_R$, for some $K_R \in \mathbb{N}$ and the constants $C_1, C_2$ are independent of $R$. Note that 
		\begin{align*} 
			\lim_{t \to \infty} \left(\frac{|x - t\xi|}{|t\xi|}\right)^{-\frac{n-1}{p(p-1)}} = 1, \quad \lim_{t \to \infty} (t - |x - t\xi|) = \left<x,\xi\right>
		\end{align*} 
		converges locally uniformly in $\Rn$, so that we have
		\begin{align}\label{ceqn3}
			v_k(x) \asymp e^{\al \left<x,\xi\right>}
		\end{align}
		in $B_R$ for all $k \ge K_R$ with constants of asymptotics independent of $R$. Applying $C^{1,\alpha}$-estimate \cite{DiBe}, \cite{Tol84} on \eqref{ceqn1} we have $[v_k]_{C^{1,\alpha}(B_R)} \le C_R$ for some $C_R$ depending only on $R > 0$. Letting $R \to \infty$ and passing through a subsequence (which we continue to index with $k$) we get $v_k \to v_\infty$ in $C^1_{\text{loc}}(\Rn)$ with $v_ \infty$ satisfying
		\begin{align}\label{ceqn4}
			& -\Delta_p v_\infty = -\lambda v_\infty^{p-1} \text{ in } \Rn \\
			& c_1 e^{\al \left<x,\xi\right>} \le v_\infty(x) \le c_2 e^{\al \left<x,\xi\right>}
		\end{align}
		from \eqref{ceqn1} and \eqref{ceqn3}. Therefore by Theorem-\ref{1thm1} we have $v_\infty(x) = e^{\al \left<x,\xi\right>}$ as $v_{\infty}(0) = v_k(0) = 1$ for all $k \in \mathbb{N}$. Since this is true for all sequences $t_k \to \infty$ we get the desired existence of limit \eqref{1eqnB} and $P(x,\xi) = e^{\al \left<x,\xi\right>}$.
	\end{proof}
	\smallskip
	\noindent We now present a proof of Theorem-\ref{1thmhardy} using essentially the rescaling and rotational invariance of the problem. The proof closely parallels the main ideas used in the proof of Theorem-\ref{1thm1}.
	\begin{proof}[Proof of Theorem-\ref{1thmhardy}] Note that \eqref{peqn1} is invariant under rescaling and orthogonal transformations i.e., $u_R(x) := u(Rx)$ and $u_A(x) := u(Ax)$ where, $A \in \mathcal{O}(n)$ is an orthogonal matrix, also solves the equation. We may linearize equation \eqref{peqn1} by differentiating with respect to $R$ or by differentiating against a one parameter family of rotations about the origin in $\Rn \setminus \{0\}$ away from the critical point set of $u$. In this case we note that the formal linearized operator corresponding to $-\Delta_{p,a}$ is given by
		\begin{align}\label{peqnlin0}
			L_{u}(\varphi) := \divg\left(|x|^{-ap}\left(|\nabla u|^{p-2} \nabla \varphi + (p-2)|\nabla u|^{p-4}\left<\nabla u, \nabla \varphi\right>\nabla u \right)\right)
		\end{align}
		which is a strictly positive definite elliptic operator whenever $|\nabla u| > 0$ in $\Rn \setminus \{0\}$.
		\\ Note that $V_0(u) := \left.\frac{d}{dR}\right\vert_{R = 1} u_R(x) = \left<x,\nabla u\right>$, so that from \eqref{peqn1} and \eqref{peqnlin0} we get $V_0(u) = \left<x,\nabla u\right>$ solves
		\begin{align}\label{peqn2}
			-L_u\left(V_0(u)\right) = (p-1)\frac{\mu}{|x|^{(a+1)p}}u^{p-2} V_0(u)
		\end{align}
		weakly in $\Rn \setminus (\{0\} \cup \mathcal{C}_u)$ where, $\mathcal{C}_u = \{\nabla u = 0\}$. 
		\\\\ Let us denote by $M^{(i_0,j_0)} := (e_{j_0} \otimes e_{i_0} - e_{i_0} \otimes e_{j_0}) = \left(m^{(i_0,j_0)}_{i,j}\right)_{n \times n}$ for $1 \le i_0 < j_0 \le n$, the matrix with entries $m_{i_0, j_0}^{(i_0,j_0)} = -1$, $m_{j_0, i_0}^{(i_0,j_0)} = 1$ and $m_{i,j}^{(i_0,j_0)} = 0$ otherwise. Note that $\left(M^{(i_0,j_0)}\right)^2 = -\left(e_{i_0} \otimes e_{i_0} + e_{j_0} \otimes e_{j_0}\right)$. We consider the one parameter family of rotation matrices $A^{(i_0,j_0)}(t) := \exp\left(M^{(i_0,j_0)}t\right)$ for $1 \le i_0 < j_0 \le n$ which are rotations in $e_{i_0}$-$e_{j_0}$ plane, given by $$A^{(i_0,j_0)}(t) = I_n + M^{(i_0,j_0)} \sin t + \left(M^{(i_0,j_0)}\right)^2 (1 - \cos t) = \left(a_{i,j}^{(i_0,j_0)}(t)\right)_{n \times n}$$ where, $a_{i_0,i_0}^{(i_0,j_0)}(t) = a_{j_0,j_0}^{(i_0,j_0)}(t) = \cos t$, $a_{i_0,j_0}^{(i_0,j_0)}(t) = -\sin t$, $a_{j_0,i_0}^{(i_0,j_0)}(t) = \sin t$ and $a_{i,j}^{(i_0,j_0)}(t) = \delta_{i,j}$ for $i,j \notin \{i_0,j_0\}$.
		\\ For $1 \le i_0 < j_0 \le n$ we have $$V_{i_0,j_0}(u) := \left.\frac{d}{dt}\right\vert_{t = 0} u(A^{(i_0,j_0)}(t)x) = \left<M^{(i_0,j_0)}x, \nabla u\right> =\left(x_{i_0}\frac{\partial u}{\partial x_{j_0}} - x_{j_0}\frac{\partial u}{\partial x_{i_0}}\right)$$ solves the equation
		\begin{align}\label{peqn3}
			-L_u\left(V_{i_0,j_0}(u)\right) = (p-1)\frac{\mu}{|x|^{(a+1)p}}u^{p-2} V_{i_0,j_0}(u)
		\end{align}
		weakly in $\Rn \setminus (\{0\} \cup \mathcal{C}_u)$. We remark that we may interpret \eqref{peqn2} and \eqref{peqn3} classically in $\Rn \setminus (\{0\} \cup \mathcal{C}_u)$ as by standard regularity theory the solution $u$ is smooth away from $\mathcal{C}_u$.
		\\ We note that the vector fields $V_{i_0,j_0} := x_{i_0}\frac{\partial }{\partial x_{j_0}} - x_{j_0}\frac{\partial }{\partial x_{i_0}}$ for $1 \le i_0 < j_0 \le n$ and $V_0 := \sum_{i=1}^n x_i\frac{\partial }{\partial x_i}$ span the tangent space of $\Rn \setminus \{0\}$ at any point and $V_{i_0,j_0} \perp V_0$ for $1 \le i_0 < j_0 \le n$.
		\\ Therefore combining \eqref{peqn1}, \eqref{peqn2} and \eqref{peqn3} we have
		\begin{align}\label{peqnlin}
			-L_u\left(V(u) + cu\right) = (p-1)\frac{\mu}{|x|^{(a+1)p}}u^{p-2} (V(u) + cu)
		\end{align}
		weakly in $\Rn \setminus (\{0\} \cup \mathcal{C}_u)$, for any $c \in \R$ and $V \in \operatorname{span}\{V_0, V_{i_0,j_0}: 1 \le i_0 < j_0 \le n\}$.
		\\\\ Now we prove the following gradient estimate for the solutions of problem \eqref{peqn1}:
		\begin{align}\label{peqng}
			\sup_{\Rn \setminus \{0\}} \frac{|x| |\nabla u(x)|}{u(x)} := \tilde{\gamma} < \infty.
		\end{align}
		To see this we note that $v_R(x) := R^{\gamma}u_R(x) = R^{\gamma}u(Rx)$ satisfies the problem \eqref{peqn1} for all $R > 0$. In particular $w_R := \log v_R$ satisfies the equation 
		\begin{align}\label{peqn4}
			& -\Delta_{p,a} w_R = \frac{\mu}{|x|^{(a+1)p}} + \frac{(p-1)}{|x|^{ap}}|\nabla w_R|^p \\ 
			\label{peqn4a}
			& - \gamma\log |x| + c_1 \le w_R(x) \le - \gamma\log |x| + c_2
		\end{align}
		in $\Rn \setminus \{0\}$ for some constants $c_1, c_2 \in \R$.
		\\\\ Now fix a $\rho > 1$ and consider the equation \eqref{peqn4} in the annular set $A(\rho) := B_{\rho}(0) \setminus B_{1/\rho}(0)$. Note that by \eqref{peqn4a} the $w_R$ is uniformly bounded in $A(\rho)$ for all $R > 0$ with bounds depending only on $\rho$. Therefore we may apply $C^{1,\alpha}$-estimate of \cite{DiBe}, \cite{Tol84} to $w_R$ in the annular sets $A(\rho)$ to get $[w_R]_{C^{1,\alpha}(A(\rho))} \le C(\rho)$ for some constant $C(\rho) >0$ depending on the parameters $n,p, \mu, \gamma, c_1, c_2$ and $\rho$. In particular we have the gradient estimate 
		\begin{align}\label{peqn5}
			\sup_{A(\rho)} |\nabla w_R(x)| = \sup_{A(\rho)} \frac{|\nabla v_R(x)|}{v_R(x)} = \sup_{A(\rho)}\frac{R|\nabla u(Rx)|}{u(Rx)} \le [w_R]_{C^{1,\alpha}(A(\rho))} \le C(\rho)
		\end{align}
		uniformly in $R > 0$. This in turn implies the gradient estimate \eqref{peqng}.
		\\\\ Let us fix $i_0, j_0$ with $i_0 < j_0$ and consider the vector field $V := \kappa_1 V_0 + \kappa_2 V_{i_0,j_0}$ for some given $\kappa_1 > 0$ and $\kappa_2 \in \R$. Note that $V$ is non-vanishing as $V_0$ is non-vanishing and $V_{i_0,j_0} \perp V_0$. Owing to the gradient estimate \eqref{peqng} we have
		\begin{align}\label{peqn6}
			-\infty < \beta:= \inf_{\Rn \setminus \{0\}} \frac{V(u)}{u} \le \sup_{\Rn \setminus \{0\}} \frac{V(u)}{u} := \alpha < +\infty.
		\end{align}
		We split the rest of the proof into two key steps.
		\\\\ \textbf{Step-I:} We claim that in \eqref{peqn6} we have
		\begin{align}\label{peqn6a} 
			\alpha, \beta \in \{0, -\kappa_1\gamma\}
		\end{align} 
		where $\kappa_1$ is as in the definition of $V$. We prove the claim for $\alpha = \sup_{\Rn \setminus \{0\}} \frac{V(u)}{u}$, the proof for $\beta = \inf_{\Rn \setminus \{0\}} \frac{V(u)}{u}$ follows similarly. If $\alpha = 0$ then we are through. Suppose to the contrary $\alpha \neq 0$, then we show that $\alpha = -\kappa_1 \gamma$ when $\gamma \neq 0$ or arrive at a contradiction when $\gamma = 0$. 
		\\ Let $\{x_k\}_{k \in \mathbb{N}}$ be a sequence of points in $\Rn \setminus \{0\}$ such that $\frac{V(u)}{u}(x_k) \to \alpha$ as $k \to \infty$. Let us denote the sequence of radii $R_k := |x_k|$ and the points $\xi_k := \frac{x_k}{|x_k|} \in \mathbb{S}^{n-1}$. Then note that 
		\begin{align}\label{peqn7}
			\sup_{\Rn \setminus \{0\}} \frac{V(u)}{u} = \sup_{\Rn \setminus \{0\}} \frac{V(u_R)}{u_R} = \sup_{\Rn \setminus \{0\}} V(w_R) = \sup_{\substack{\xi \in \mathbb{S}^{n-1}, \\ R > 0}} \frac{V(u_{R})(\xi)}{u_{R}(\xi)} = \sup_{\substack{\xi \in \mathbb{S}^{n-1}, \\ R > 0}} V(w_{R})(\xi) = \alpha 
		\end{align} 
		for all $R > 0$ and in particular
		\begin{align}\label{peqn7a}
			\frac{V(u)(x_k)}{u(x_k)} = \frac{V(u_{R_k})(\xi_k)}{u_{R_k}(\xi_k)} = V(w_{R_k})(\xi_k) \to \alpha 
		\end{align}
		as $k \to \infty$. Using the estimate from \eqref{peqn5}, $[w_{R_k}]_{C^{1,\alpha}(A(\rho))} \le C(\rho)$ for $\rho > 1$ we may extract a subsequence (which we continue to index with $k$) such that $\xi_k \to \xi \in\mathbb{S}^{n-1}$, $w_{R_k} \to w_{\infty}$ in $C^1_{\text{loc}}(\Rn \setminus \{0\})$ and $w_\infty$ satisfies 
		\begin{align}\label{peqn7b}
			& -\Delta_{p,a} w_\infty = \frac{\mu}{|x|^{(a+1)p}} + \frac{(p-1)}{|x|^{ap}}|\nabla w_\infty|^p \\ 
			\label{peqn7c}
			& - \gamma\log |x| + c_1 \le w_\infty(x) \le - \gamma\log |x| + c_2.
		\end{align} 
		Hence, $u_{\infty} := e^{w_\infty}$ satisfies the problem \eqref{peqn1}. Also from the definition, $C^1_{\text{loc}}(\Rn \setminus \{0\})$ convergence of $w_{R_k} \to w_{\infty}$ and \eqref{peqn7} we have 
		\begin{align}\label{peqn8i}
			\sup_{\Rn \setminus \{0\}} \frac{V(u_\infty)}{u_\infty} = \sup_{\Rn \setminus \{0\}} V(w_\infty) \le \alpha
		\end{align}
		and in particular from \eqref{peqn7a} we have $\frac{V(u_\infty)}{u_\infty}(\xi) = V(w_\infty)(\xi) = \alpha$. 
		\\ Therefore, from \eqref{peqnlin} and \eqref{peqn8i} we have $(\alpha u_\infty - V(u_\infty))$ is a non-negative weak solution to the linearized equation 
		\begin{align}\label{peqn8}
			-L_{u_\infty}\left(\alpha u_\infty - V(u_\infty)\right) = (p-1)\frac{\mu}{|x|^{(a+1)p}}u_\infty^{p-2} (\alpha u_\infty - V(u_\infty))
		\end{align}
		in $\Rn \setminus (\{0\} \cup \mathcal{C}_{u_\infty})$. Since $(\alpha u_\infty - V(u_\infty))(\xi) = 0$, $\alpha \neq 0$ by our assumption, $u_\infty > 0$ and $V$ is non-vanishing, we have $|\nabla u_\infty| > 0$ in an open neighborhood of $\xi$ in $\Rn \setminus \{0\}$. By applying Harnack inequality (\cite{Serrin70}, \cite{T}) and a connectedness argument as before we have $V(u_\infty) \equiv \alpha u_\infty$ in $\Rn \setminus \{0\}$ i.e., $V(w_\infty) \equiv \alpha$ in $\Rn \setminus \{0\}$. Therefore, $w_\infty$ grows linearly along the integral curves of $V$.
		\\\\ Let $y(t)$ be an integral curve of the vector field $V = (\kappa_1 V_0 + \kappa_2 V_{i_0,j_0})$ in $\Rn \setminus \{0\}$ i.e., $y'(t) = \left.V\right\vert_{y(t)}$. We denote this system of linear ordinary differential equations as 
		\begin{align}\label{peqn8a}
			y'(t) = Ay(t)
		\end{align}
		where, $A = \kappa_1 I_n + \kappa_2 M^{(i_0,j_0)}$, i.e., $y_{i_0}'(t) = \kappa_1 y_{i_0}(t) - \kappa_2 y_{j_0}(t)$, $y_{j_0}'(t) = \kappa_1 y_{j_0}(t) + \kappa_2 y_{i_0}(t)$ and $y_i'(t) = \kappa_1 y_i(t)$ for $i \neq i_0, j_0$. Then the curve $y(t)$ with initial point $y(0) := y_0 \in \Rn \setminus \{0\}$ is given by 
		\begin{align}
			y(t) = \exp\left(At\right)y_0 = e^{\kappa_1 t} \exp\left(\kappa_2 M^{(i_0,j_0)} t\right) y_0.
		\end{align}
		Using the fact that $\frac{d}{dt} w_\infty(y(t)) = V(w_\infty)(y(t)) \equiv \alpha$ we get $$w_{\infty}(y(t)) - w_\infty(y_0) = \alpha t$$ for all $t \in \R$. Therefore, $u_\infty(y(t)) = e^{\alpha t}u_\infty(y_0)$ and from the point-wise estimate on $u_\infty$ we have 
		\begin{align} 
			e^{\alpha t}u_\infty(y_0) = u_\infty(y(t)) \asymp |y(t)|^{-\gamma} = \left|e^{\kappa_1 t} \exp\left(\kappa_2 M^{(i_0,j_0)} t\right) y_0\right|^{-\gamma} \asymp e^{-\gamma \kappa_1 t}
		\end{align} 
		for all $t \in \R$. Hence, $\alpha = -\gamma \kappa_1$ when $\gamma \neq 0$ and leads to a contradiction when $\gamma = 0$ (as $\alpha \neq 0$), proving the claim.
		\\\\ \textbf{Step-II:} From previous step \eqref{peqn6a} when $\gamma \ge 0$ we have $\alpha \le 0$ i.e., $$V(u) = \kappa_1 V_0(u) + \kappa_2 V_{i_0,j_0}(u) \le 0$$ and when $\gamma < 0$ we have $\beta \ge 0$ i.e., $$V(u) = \kappa_1 V_0(u) + \kappa_2 V_{i_0,j_0}(u) \ge 0$$ for all $\kappa_1 > 0$ and $\kappa_2 \in \R$. In either case letting $\kappa_1 \to 0^+$ and setting $\kappa_2 = \pm 1$ we get $V_{i_0,j_0}(u) \equiv 0$ in $\Rn \setminus \{0\}$ for $1 \le i_0 < j_0 \le n$. Therefore, $\left<\nabla u, \xi\right> = 0$ for all $\xi \in V_0^\perp$ and hence $u$ must be a radial function. By abuse of notation we denote $u(x) = u(r)$ where, $r = |x|$. Then if $\gamma \ge 0$ we have $V_0(u) = ru'(r) \le 0$ i.e., $u$ is non-increasing and if $\gamma < 0$ we have $V_0(u) = ru'(r) \ge 0$ i.e., $u$ is non-decreasing. Then the equation \eqref{peqn1} can be written as 
		\begin{align}\label{peqn9}
			& -\left(r^{n-1-ap}|u'(r)|^{p-2}u'(r)\right)' = \mu \, r^{n-1 - (a+1)p} u^{p-1}(r) \text{ in } \R_+ = (0,\infty) \\ 
			\label{peqn9a}
			& C_1r^{-\gamma} \le u(r) \le C_2 r^{-\gamma}.
		\end{align}
		Also from the point-wise gradient estimate \eqref{peqng} we have $r|u'(r)| \le \tilde{\gamma}u(r)$. 
		\\\\ First we consider the case $\mu = 0$ so that either $\gamma = 0$  or  $\gamma = \frac{n - (a+1)p}{p-1}$ (when $n - (a+1)p \neq 0$). We have $\left(r^{n-1-ap}|u'(r)|^{p-2}u'(r)\right)' = 0$ in $\R_+$ and consequently either $u'(r) \equiv 0$ in which case $u$ is a constant function or $|u'(r)| = Cr^{-\frac{n - (a+1)p}{p-1} - 1}$ for some $C > 0$. If $n - (a+1)p = 0$ then $|u'(r)| = \frac{C}{r}$ can be ruled out by the point-wise bounds \eqref{peqn9a}. So the remaining possibility is $u(r) = Cr^{-\frac{n - (a+1)p}{p-1}}$ for some $C > 0$.
		\\\\ Otherwise when $\mu \neq 0$ we have $\gamma \neq 0, \frac{n-(a+1)p}{p-1}$ as well. We use the point-wise bounds \eqref{peqn9a} on RHS of \eqref{peqn9}, the point-wise gradient estimate \eqref{peqng} $$r|u'(r)| \le \tilde{\gamma}u(r) \implies r^{n - 1 -ap}|u'(r)|^{p-1} \le C'r^{(p-1)\left(\frac{n - (a+1)p}{p-1} - \gamma\right)}$$ and integrate \eqref{peqn9} from $0$ to $r$ when $\gamma < \frac{n-(a+1)p}{p-1}$ or from $r$ to $\infty$ when $\gamma > \frac{n-(a+1)p}{p-1}$ to get the point-wise estimate for the gradient $|u'(r)| \asymp r^{-\gamma - 1}$ in $(0,\infty)$. This in particular implies $\alpha := \sup_{\Rn \setminus \{0\}} \frac{V_0(u)}{u} = \sup_{r > 0} \frac{ru'(r)}{u(r)}$ and $\beta := \inf_{\Rn \setminus \{0\}} \frac{V_0(u)}{u} = \inf_{r > 0} \frac{ru'(r)}{u(r)}$ in \eqref{peqn6} are never equal to $0$. Therefore by the claim in Step-I for $V = V_0$ we must have $\alpha = \beta = -\gamma$. Hence $\frac{ru'(r)}{u(r)} \equiv -\gamma$ in $(0,\infty)$ i.e., $u(r) = Cr^{-\gamma}$ for some $C >0$. \end{proof}
	
	\section{Appendix}
	\subsection{Proof of the sharp gradient estimate for eigenfunctions}
	In this section we present a proof of the sharp gradient Lemma-\ref{lemmagrad}.
	\begin{proof}[Proof of Lemma-\ref{lemmagrad}]
		By translation invariance of \eqref{E1}, $C^{1,\alpha}$-estimate \cite{DiBe}, \cite{Tol84} and Harnack inequality \cite{T} we have
		$$|\nabla v|(x_0) \le \sup_{B_{1/2}(x_0)} |\nabla v| \le C\sup_{B_{1}(x_0)} v \le c(n,p,\lambda) v(x_0)$$ for all $x_0 \in \mathbb{R}^n$ for some universal constant $c(n,p,\lambda) > 0$ i.e., $|\nabla \log v| \le c(n,p,\lambda)$ in $\Rn$.
		\\ We set the notations $w := -(p-1)\log v$ and the constant $\kappa^{p/2} := (p-1)^{p-1}\lambda$. Then $w$ satisfies the equation
		\begin{align}\label{N2eqn7}
			\Delta_p w = -\lambda (p-1)^{p-1} + |\nabla w|^p
		\end{align}
		and further we use the notation $f = |\nabla w|^2$. 
		\\ \textit{Claim}: $f \le \kappa$ in $\Rn$. 
		\\ Suppose to the contrary that $\{f > \kappa\}$ is non-empty. Following the notation of linearized $p$-Laplacian from \eqref{2eqn1} and \cite[Lemma-2.1]{SW14} we have the point-wise Bochner identity
		\begin{align}\label{N2eqn8}
			L_w(f) = 2f^{\frac{p}{2}-1} w_{ij}^2 + \left(\frac{p}{2} - 1\right)|\nabla f|^2f^{\frac{p}{2}-2} + pf^{\frac{p}{2}-1}\left<\nabla w, \nabla f\right>
		\end{align} wherever $f = |\nabla w|^2 > 0$. In this case we note that $\operatorname{Ric} = 0$ in the identity. By standard Elliptic regularity theory we know $w, f \in C^\infty$ wherever $f = |\nabla w| > 0$. By Sard's theorem almost all level sets $\{f = c\}$ for $c > 0$ are regular and smooth. The outward unit normal to $\partial \{f > c\} = \{f = c\}$ is given by $\nu := -\frac{\nabla f}{|\nabla f|}$ for these levels. 
		\\ Let $\delta > 0$ be such that $\{f = \kappa + \delta\}$ is a regular level set. Denote $\omega := \left(f - (\kappa + \delta)\right)_+$ and the support of $\omega$ as $\Omega := \left\{f > \kappa + \delta \right\}$, so that $\omega \in C^\infty(\overline{\Omega})$. The first step is to show that there exists positive constants $a,b > 0$ depending only on the parameters $n, p, \delta$ such that
		\begin{align}\label{N2eqn9}
			\int_{\Rn} L_w(\varphi)\omega \,dx \ge \int_{\Rn} \varphi (a\omega - b|\nabla \omega|)\,dx
		\end{align}
		for all non-negative test functions $\varphi \in C_c^{\infty}(\Rn)$. Applying integration by parts twice on LHS of \eqref{N2eqn9} we have
		\begin{align}\label{N2eqn10}
			\int_{\Rn} L_w(\varphi)\omega \,dx &= \int_{\Omega} L_w(\omega) \varphi \,dx + \int_{\partial \Omega} f^{\frac{p}{2}-1}\left<A(\nabla \varphi),\nu\right> \omega\,d\sigma - \int_{\partial \Omega} f^{\frac{p}{2}-1}\left<A(\nabla \omega),\nu\right> \varphi \,d\sigma
		\end{align}
		where, $\nu = -\frac{\nabla f}{|\nabla f|}$ is the outward unit normal to $\partial\Omega$. By a abuse of notation we will also denote $\nu = -\frac{\nabla \omega}{|\nabla \omega|}$ on $\partial \Omega$. Note that the first boundary term on RHS of \eqref{N2eqn10} vanishes as  $\omega = 0$ on $\partial \Omega$ and as for the second boundary term in \eqref{N2eqn10} we have $$-\int_{\partial \Omega} f^{\frac{p}{2}-1}\left<A(\nabla \omega),\nu\right> \varphi \,d\sigma = \int_{\partial \Omega} f^{\frac{p}{2}-1}\left<A(\nabla \omega),\frac{\nabla \omega}{|\nabla \omega|}\right> \varphi \,d\sigma \ge 0.$$ Therefore
		\begin{align}\label{N2eqn11}
			\int_{\Rn} L_w(\varphi)\omega \,dx &\ge \int_{\Omega} L_w(\omega) \varphi \,dx
		\end{align}
		for all non-negative test functions $\varphi \in C_c^{\infty}(\Rn)$. 
		\\ It suffices to show that $L_w(\omega) \ge (a\omega - b|\nabla \omega|)$ in $\Omega$. Using an orthonormal frame of coordinates $\{e_1, \cdots, e_n\}$ in $\Rn$ such that $e_n = \frac{\nabla w}{|\nabla w|}$ we can rewrite equation \eqref{N2eqn7} as 
		\begin{align}\label{N2eqn12}
			& \Delta_p w = |\nabla w|^{p-2}\left(\Delta u + (p-2)|\nabla w|^{-2}\left<(\nabla^2 w) \nabla w, \nabla w\right>\right) = -\lambda (p-1)^{p-1} + |\nabla w|^p \nn
			\\ \implies & f^{\frac{p}{2}-1}\left(w_{11} + \cdots + w_{nn} + (p-2)w_{nn}\right) = f^{\frac{p}{2}} - \lambda (p-1)^{p-1} \nn
			\\ \implies & w_{11} + \cdots + w_{n-1,n-1} = f - f^{1 - \frac{p}{2}}\lambda (p-1)^{p-1} - (p-1) w_{nn}.
		\end{align}
		\\ Then using Cauchy-Schwarz inequality to the first term on the RHS of equation \eqref{N2eqn8} followed by the substitution of the identity \eqref{N2eqn12} we have
		\begin{align}\label{N2eqn13}
			\sum_{1 \le i,j \le n} w_{ij}^2 &\ge \frac{(w_{11} + \cdots + w_{n-1, n-1})^2}{n-1} + \sum_{1 \le i \le n} w_{in}^2 \nn 
			\\ &= \frac{(f - f^{1 - \frac{p}{2}}\lambda (p-1)^{p-1} -(p-1)w_{nn})^2}{n-1} + \sum_{1 \le i \le n} w_{in}^2.
		\end{align}
		Note that $\nabla f = 2|\nabla w|(w_{1n}, \cdots, w_{nn})^T = 2f^{1/2}(w_{1n}, \cdots, w_{nn})^T$. Therefore, $\left<\nabla f, \nabla w\right> = 2fw_{nn}$. Also, $$\sum_{1 \le i \le n} w_{in}^2 = \frac{|\nabla f|^2}{4f}.$$ Therefore, using these we rewrite the RHS of inequality \eqref{N2eqn13} as
		\begin{align}\label{N2eqn14}
			\sum_{1 \le i,j \le n} w_{ij}^2 &\ge \frac{(f - f^{1 - \frac{p}{2}}\lambda (p-1)^{p-1})^2}{n-1} - 2\frac{p-1}{n-1}w_{nn}\left(f - f^{1 - \frac{p}{2}}\lambda (p-1)^{p-1} \right) + \sum_{1 \le i \le n} w_{in}^2 \nn 
			\\ & =  \frac{(f - f^{1 - \frac{p}{2}}\lambda (p-1)^{p-1})^2}{n-1} - \frac{p-1}{n-1}\left(f - f^{1 - \frac{p}{2}}\lambda (p-1)^{p-1} \right) \frac{\left<\nabla f, \nabla w\right>}{f} + \frac{|\nabla f|^2}{4f}.
		\end{align}
		Using inequality \eqref{N2eqn14} in the identity \eqref{N2eqn8} we get
		\begin{align}\label{N2eqn15}
			L_w(\omega) &\ge 2f^{1 - \frac{p}{2}} \frac{(f^{\frac{p}{2}} - \lambda (p-1)^{p-1})^2}{n-1} - 2\frac{p-1}{n-1}f^{-1}\left(f^{\frac{p}{2}} - \lambda (p-1)^{p-1} \right)\left<\nabla f, \nabla w\right> \nn 
			\\ & \quad + \left(\frac{p-1}{2}\right)\frac{|\nabla f|^2}{4f} + pf^{\frac{p}{2} - 1}\left<\nabla w, \nabla f\right>.
		\end{align}
		Note that we have $0 < \kappa + \delta \le f = |\nabla w|^2 \le C(n,p,\lambda)$, $f = \omega + \kappa + \delta$ and $\nabla f = \nabla \omega$ in $\Omega$. Using the mean value theorem we may write
		\begin{align}\label{N2eqn16}
			(f^{\frac{p}{2}} - \lambda (p-1)^{p-1})^2 = \frac{p^2}{4} \xi^{p - 2} (f - \kappa)^2 \ge c_0 \omega 
		\end{align} where, $\kappa < \xi < f$ and $c_0$ is a constant depending on $n, p , \kappa, \delta$. Finally using this in \eqref{N2eqn15} together with the point-wise bounds on $f$, $|\nabla w|$ in $\Omega$ we have
		\begin{align}\label{N2eqn17}
			L_w(\omega) \ge a\omega - b|\nabla \omega| \text{ in } \Omega
		\end{align}
		and consequently together with the inequality \eqref{N2eqn11} we get \eqref{N2eqn9}.
		\\\\ Testing \eqref{N2eqn9} with $\varphi^2\omega^q$ for some non-negative test function $\varphi$ in $\Rn$ we get the inequality
		\begin{align}\label{N2eqn18}
			-\int_{\Omega} f^{\frac{p}{2}-1}\left<A(\nabla (\varphi^2 \omega^q)), \nabla \omega\right> \,dx \ge \int_{\Omega} \left( a\varphi^2 \omega^{q+1} - b\varphi^2 \omega^q |\nabla \omega| \right) \,dx.
		\end{align}
		Rearranging the terms in \eqref{N2eqn18} and using the inequality $\left<A(\nabla \omega), \nabla \omega\right> \ge \min\{1,p-1\}|\nabla \omega|^2$ we have
		\begin{align}\label{N2eqn19}
			a \int_{\Omega} \varphi^2 \omega^{q+1}\,dx &\le b\int_{\Omega} \varphi^2 \omega^q |\nabla \omega| \,dx + c_1 \int_{\Omega} \varphi \omega^q |\nabla \varphi|  |\nabla \omega| \,dx - c_2q \int_{\Omega} \varphi^2 \omega^{q-1} |\nabla \omega|^2 \,dx.
		\end{align}
		Again using Cauchy-Schwarz inequality we have 
		\begin{align}\label{N2eqn20}
			a \int_{\Omega} \varphi^2 \omega^{q+1}\,dx &\le \epsilon \int_{\Omega} \varphi^2 \omega^{q+1} \,dx + \frac{b}{4\epsilon}\int_{\Omega} \varphi^2 \omega^{q-1} |\nabla \omega|^2\,dx \nn \\ & \quad  + \epsilon \int_{\Omega} |\nabla \varphi|^2 \omega^{q+1} \,dx + \frac{c_1}{4\epsilon} \int_{\Omega} \varphi^2 \omega^{q-1}|\nabla \omega|^2 \omega^{q-1} \,dx \nn 
			\\ & \quad - c_2q \int_{\Omega} \varphi^2 \omega^{q-1} |\nabla \omega|^2 \,dx.
		\end{align}
		Then choosing $\epsilon < a/2$ to be sufficiently small and $q > 1$ to be large enough such that $c_2q > \frac{b+c_1}{4\epsilon}$ we have
		\begin{align}\label{N2eqn21}
			a \int_{\Omega} \varphi^2 \omega^{q+1}\,dx &\le 2\epsilon \int_{\Omega} |\nabla \varphi|^2 \omega^{q+1} \,dx.
		\end{align}
		Choosing $\varphi \equiv 1$ in $B_k$ such that $|\nabla \varphi| \le 2$ in $B_{k+1}\setminus B_k$ for positive integers $k$ we successively get
		\begin{align}\label{N2eqn22} 
			a \int_{B_k} \omega^{q+1} \,dx \le 4\epsilon \int_{B_{k+1}\setminus B_k} \omega^{q+1} \,dx \implies \frac{c}{\epsilon} \int_{B_k} \omega^{q+1} \,dx \le \int_{B_{k+1}} \omega^{q+1} \,dx.
		\end{align}
		Iterating this for $k \ge 1$ we get 
		\begin{align}\label{N2eqn23} 
			\int_{B_R} \omega^{q+1} \,dx \ge Ce^{R \log \frac{c}{\epsilon}}
		\end{align} for all $R$ sufficiently large. However, $\omega$ is a bounded function and the volume of the ball $|B_R| = |B_1|R^n$ only grows polynomialy. This leads to a contradiction unless $\omega \equiv 0$, i.e., $f \le \kappa$ in $\Rn$. This completes the proof of the sharp gradient estimate from above. \end{proof}

	\bibliographystyle{plain}
	\bibliography{Ramyabib.bib}
\end{document}